\documentclass{amsart}

\usepackage{amsmath}
\usepackage{amsfonts}
\usepackage{amsthm, thmtools}
\usepackage{hyperref, cleveref}
\usepackage{verbatim}
\usepackage{graphicx}
\usepackage{tikz}
\usepackage{url}

\newcommand{\R}{\mathbb{R}}

\newcommand{\D}{\mathbb{D}}
\newcommand{\set}[1]{\left\lbrace #1 \right\rbrace}
\newcommand{\abs}[1]{\left\lvert #1 \right\rvert}
\newcommand{\norm}[1]{\left\lVert #1 \right\rVert}

\newcommand{\MM}{\mathcal{M}}

\newtheorem{theorem}{Theorem}[section]
\newtheorem*{theorem*}{Theorem}
\newtheorem{proposition}[theorem]{Proposition}
\newtheorem{lemma}[theorem]{Lemma}
\theoremstyle{definition}
\newtheorem{definition}[theorem]{Definition}
\newtheorem{remark}[theorem]{Remark}
\newtheorem*{question*}{Question}

\DeclareMathOperator{\Int}{Int}
\DeclareMathOperator{\im}{Im}

\title{Rhombs inscribed in spheres bounding strictly-convex regions}
\author{Jacob Saunders}

\begin{document}

\maketitle

\section{Introduction}

Over a century old, the square peg problem asks if every Jordan curve (closed plane curve with no self-intersections) contains the vertices of a square.
As of the time of writing, the problem remains unsolved.
A survey of the progress made thus far on the square peg problem can be found in \cite{matschke}.
More generally, if some subset of $\R^n$ contains the vertices of some polytope, then the set is said to \textit{inscribe} the polytope.
A Jordan curve is the image of an embedding of the circle $S^1$ into $\R^2$, so it is natural to ask analogous questions for higher dimensional shapes inscribed in the image of an embedding of the $n$-sphere $S^n$ into $\R^{n + 1}$.
The square peg problem has two natural $n$-dimensional generalisations.
The first (and perhaps most obvious) asks if the image of each embedding of $S^n$ into $\R^{n + 1}$ inscribes an $(n + 1)$-cube, but this admits a counterexample in the form of a very flat tetrahedron in $\R^3$ as noted in \cite{kakutani}.
The second problem is similar, but focuses on inscribed regular crosspolytopes rather than cubes.
Some familiar $n$-crosspolytopes are the square ($n = 2$), and the octahedron ($n = 3$).

\begin{definition}
    Let $p \in \R^n$, $\set{v_1, \dots, v_n}$ be a set of orthonormal vectors and let $\lambda_1, \dots, \lambda_n \in \R_{ > 0}$.
    An $n$-\textit{rhomb} is the convex hull of the points $p \pm \lambda_i v_i$ (which are the vertices of the rhomb).
    If $\lambda_1 = \cdots = \lambda_n$ (the $n$-rhomb is regular) then we refer to the $n$-rhomb as an $n$-\textit{crosspolytope}.
\end{definition}

We can therefore associate to each $n$-rhomb an orthonormal basis $\set{v_1, \dots, v_n}$ of $\R^n$, and we call this the \textit{direction} of the $n$-rhomb.
Understanding how $(n + 1)$-rhombs are inscribed in the image of an embedding of an $n$-sphere into $\R^{n + 1}$ may offer some insight into the regular case.

\begin{figure}[h]
    \centering
    \includegraphics[width=\textwidth]{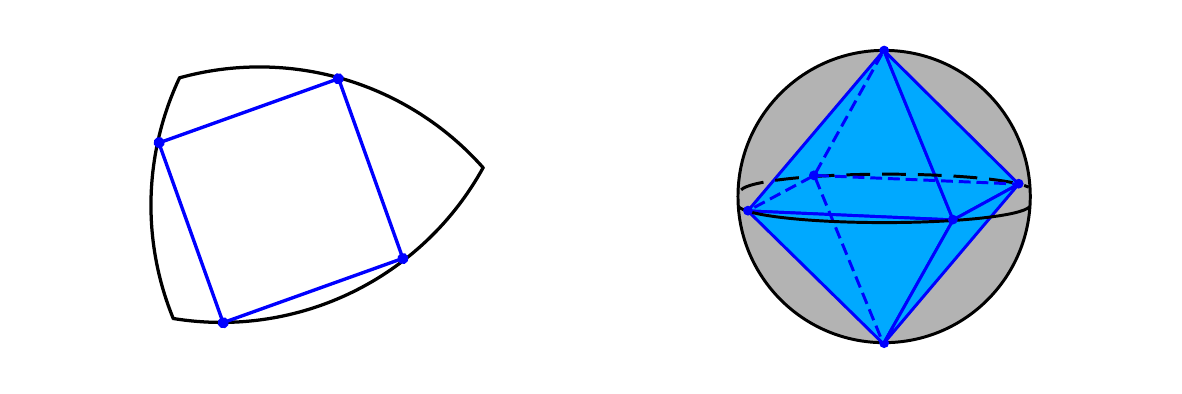}
    \caption{A square inscribed in a Jordan curve and a crosspolytope inscribed in the 2-sphere.}
\end{figure}

One of the earliest results relating to the square peg problem is due to Emch \cite{emch}, and uses the fact that every smooth Jordan curve that bounds a strictly-convex region inscribes a unique rhombus in every direction.
The lengths of the diagonals of the rhombus vary continuously when this direction is rotated, meaning one can use an intermediate value theorem argument to show that every such curve inscribes a square.
In a paper of Pucci \cite{pucci} it is claimed that every $n$-sphere bounding a convex region in $\R^{n + 1}$ inscribes an $(n + 1)$-rhomb in every direction.
A later paper of Hadwiger, Larman and Mani \cite{hyperrhombs} gives a proof of the claim for smoothly embedded $n$-spheres bounding strictly-convex regions, but provides a counterexample for $n \geq 2$ which is (non-strictly) convex, but not smooth.
\\~\\
A recent paper of Fung \cite{fung} proves that every Jordan curve inscribes uncountably many rhombi, and an integral result in the paper is that every Jordan curve with no \textit{special corners} in a given direction inscribes a rhombus in this direction.
We introduce an $n$-dimensional generalisation of the notion of a special corner in \autoref{defn-special-corner}.
Fung's technique equates to splitting a Jordan curve into a pair of paths and constructing medians i.e.\ the sets of midpoints of these paths.
It is shown that these medians intersect, corresponding to an inscribed rhombus in the desired direction.
\\~\\
The aim of this paper is to give a condition weaker than smoothness which guarantees that the image of a strictly-convex embedding of an $n$-sphere into $\R^{n + 1}$ inscribes an $(n + 1)$-rhomb in a given direction.
The condition comes in two parts: a lack of special corners with respect to a basis and another part called \textit{regularity with respect to a basis}, which we give the definition for in \autoref{defn-regularity}.
We will prove

\begin{theorem*}\label{main-theorem}
    Let $B$ be the image of a strictly-convex embedding of $\D^{n + 1}$ into $\R^{n + 1}$ with boundary $n$-sphere $S$.
    Let $\set{v_1, \dots, v_{n + 1}}$ be an orthonormal basis of $\R^{n + 1}$ such that
    \begin{itemize}
        \item $B$ is regular with respect to this basis, and
        \item $B$ has no special corners with respect to this basis.
    \end{itemize}
    Then $S$ inscribes an $(n + 1)$-rhomb with direction $\set{v_1, \dots, v_{n + 1}}$.
\end{theorem*}

Our arguments are presented in terms of the orthogonal projection of a subset of $\R^n$ to a hyperplane with normal $v \in \R^{n}$ (denoted $p_v$).
In \hyperref[section-2]{Section 2}, we show via \autoref{preimage-of-boundary-is-sphere} that a point $y \in p_v(B)$ is in the boundary of $p_v(B)$ exactly when there is only one point $x \in S$ with $p_v(x) = y$.
As a consequence, the set $D_v^0$ of points in $B$ that project to the boundary of $p_v(B)$ is homeomorphic to $S^{n - 1}$.
If a point does not project to the boundary then we deduce via \autoref{interior-pair-of-points} that there are exactly two such points $x^+, x^- \in S$ with $p_v(x^+) = p_v(x^-) = y$.
Using \autoref{top-and-bottom-open}, we show that we are able to split the $n$-sphere $S$ bounding $B$ into a pair of open sets $U_v^+$ and $U_v^-$ representing upper and lower sections of the sphere.
The section ends with a proof of the following, where $D_v^{\pm}$ denotes $D_v^0 \cup U_v^{\pm}$.

\begin{proposition}\label{median-is-disc}
    The set
    \[ \MM_{v} = \set{\frac{x^+ + x^-}{2} : x^+ \in D_v^+, x^- \in D_v^-, p_v(x^+) = p_v(x^-)} \]
    is homeomorphic to $\D^{n}$, and $\MM_v \cap S$ is homeomorphic to $S^{n - 1}$.
\end{proposition}

An orthonormal basis $\set{v_1, \dots, v_{n + 1}}$ of $\R^{n + 1}$ provides us with a family of spheres $D_i^0$ (where $D_i^0$ refers to $D_{v_i}^0$).
\hyperref[section-3]{Section 3} starts with a proof of the following result regarding how the spheres are arranged on $S$, as long as $B$ is regular with respect to the basis, and has no special corners with respect to the basis.

\begin{proposition}\label{intersections-of-spheres}
    The spheres $D_i^0$ are arranged on $S$ such that
    \begin{itemize}
        \item if $I$ is a subset of $\set{1, \dots, n + 1}$ of size $k \leq n$, then the set $\bigcap_{i \in I} D_i^0$ is homeomorphic to $S^{n - k}$,
        \item for each $j$, $\bigcap_{i \neq j} D_i^0$ is a pair of points $x^{\pm}$ such that $x^{\pm} \in D_j^{\pm}$,
        \item $\bigcap_{i} D_i^0$ is empty.
    \end{itemize}
\end{proposition}

We say some $B$ satisfying the outcomes of \autoref{intersections-of-spheres} satisfies the \textit{crosspolytope condition}, continuing \hyperref[section-3]{Section 3} by outlining how this implies the existence of a homeomorphism from an $(n + 1)$-crosspolytope to $B$ which sends the $(n - 1)$-skeleton of the crosspolytope (the set of $k$-faces of the crosspolytope where $k \leq n - 1$) to the union of the spheres $D_i^0$.
\autoref{cube-crosspolytope-balls} provides a homeomorphism from the $(n + 1)$-cube to the $(n + 1)$-crosspolytope with properties that allow us to use the Poincar{\'e}-Miranda theorem (found, for example, in \cite{poincare-miranda}) to prove the following.

\begin{proposition}\label{intersections-of-manifolds}
    Let $K_1, \dots, K_{n + 1}$ be compact subsets of $B$ such that
    \begin{itemize}
        \item $K_i \cap S =: D_i^0$ is homeomorphic to an $(n - 1)$-sphere,
        \item there exist pairs $V_i^{\pm}$ of disjoint open subsets of $B$ such that $U_i^{\pm} \subset V_i^{\pm}$ and the boundary of $V_i^{\pm}$ is contained in $K_i$, and
        \item $B$ satisfies the crosspolytope condition.
    \end{itemize}
    Then the set $\bigcap_{i} K_i$ is non-empty and is contained in the interior of $B$.
\end{proposition}

Our \hyperref[main-theorem]{Theorem} follows from this.

\section*{Acknowledgements}
The author expresses his deepest gratitude to Andrew Lobb for his supervision and guidance throughout the summer of 2021, as well as for proofreading this paper, and drawing the author's attention to \cite{fung} which inspired much of the paper.
The author also expresses his thanks to the London Mathematical Society and Durham University for awarding him funding through the London Mathematical Society's Undergraduate Research Bursary scheme.

\section{Construction of Medians}\label{section-2}

\begin{definition}
    Let $B$ be a subset of $\R^{n + 1}$.
    We call $B$ \textit{convex} if for all $x, y \in B$, all points on the straight line from $x$ to $y$ are contained in $B$.
    Furthermore, $B$ is \textit{strictly-convex} if all points on the straight line from $x$ to $y$ are contained in the interior of $B$ (except possibly $x$ and $y$ themselves).
\end{definition}

The paper \cite{hyperrhombs} deals with compact convex subsets of $\R^{n + 1}$ rather than the image of an embedding of $S^n$, but the following proposition (Proposition 5.1 in \cite{lee}) reveals that these are two instances of the same thing.

\begin{proposition}\label{characterisation-of-disc}
    If $B$ is a compact convex subset of $\R^{n + 1}$ with non-empty interior, then $B$ is homeomorphic to $\D^{n + 1}$ and $\partial B$ is homeomorphic to $\partial \D^{n + 1} = S^n$.
\end{proposition}

From now on, let $B$ be a compact strictly-convex subset of $\R^{n + 1}$ with non-empty interior.
It follows that $B$ is homeomorphic to $\D^{n + 1}$, and is bounded by an $n$-sphere $S$.
An embedding of $\D^{n + 1}$ into $\R^{n + 1}$ with strictly-convex image $B$ is referred to as a \textit{strictly-convex embedding} of $\D^{n + 1}$, and we use the term strictly-convex embedding of $S^n$ to refer to the restriction to $\partial \D^{n + 1}$ of a strictly-convex embedding of $\D^{n + 1}$.
Note that the results in this section are independent of regularity and the presence of special corners in $B$.

The projection $p_v : \R^{n + 1} \to \R$ is the quotient map of the equivalence relation
\[ x \sim y \iff x = y + \lambda v, \qquad \lambda \in \R. \]

\begin{lemma}\label{projection-is-strictly-convex}
    The set $p_v(B)$ is a compact strictly-convex subset of $\R^n$ with non-empty interior (and is therefore homeomorphic to $\D^n$).
\end{lemma}

\begin{proof}
    Note that $p_v(B)$ is compact as it is the continuous image of a compact set.
    Also $p_v(B)$ is strictly-convex; indeed, for all $x, y \in p_v(B)$, there exist corresponding $x + rv, y + sv \in B$, so $z := x + rv + \lambda(y + sv - x - rv) \in B$ for all $\lambda \in [0, 1]$.
    We have
    \[ p_v(z) = p_v(x + \lambda(y - x) + v(r + \lambda(s - r))) = x + \lambda(y - x), \]
    and $z$ is contained in the interior of $B$ by strict-convexity, so there is an open ball around $z$ contained in $B$ which projects to an open ball around $p_v(z)$ contained in $p_v(B)$.
    The set $B$ has non-empty interior, so there exists some $z$ in the interior of $B$ which is contained in an open ball contained in $B$, which projects to an open ball around $p_v(z)$ in $p_v(B)$.
    Hence $p_v(B)$ is a compact convex subset of $\R^n$ with non-empty interior and is homeomorphic to $\D^n$ by \autoref{characterisation-of-disc}.
\end{proof}

It follows that the boundary of $p_v(B)$ is homeomorphic to $S^{n - 1}$.
Let $D_v^0$ be the preimage of $\partial p_v(B)$ under $p_v : B \to \R^n$.

\begin{lemma}\label{preimage-of-boundary-is-sphere}
    The set $D_v^0$ is contained in $S$, and $p_v : D_v^0 \to \partial p_v(B)$ is a homeomorphism.
\end{lemma}

\begin{proof}
    Let $x$ be a point in the interior of $B$.
    Then $x$ is contained in an open ball contained within $B$, the image of which is an open ball around $p_v(x)$ in $p_v(B)$, so $p_v(x)$ lies in the interior of $p_v(B)$.
    Hence $S$ contains all points in $B$ which project to $\partial p_v(B)$.
    Note that $D_v^0$ is closed as it is the preimage of a closed set under the continuous function $p_v$, and therefore compact as it is a closed subset of a compact space.
    We now prove that for each $y \in \partial p_v(B)$ there exists a unique $x \in S$ with $p_v(x) = y$.
    At least one such $x$ exists as $p_v(S) = p_v(B)$.
    If this $x$ is non-unique i.e.\ some other $x'$ with the property exists then by definition of $p_v$ we also have
    \[ p_v\left( \frac{x + x'}{2} \right) = y, \]
    but by strict-convexity, $\frac{x + x'}{2}$ is contained in the interior of $B$.
    Hence there exists some open ball around $\frac{x + x'}{2}$ contained in $B$.
    Such an open ball projects to an open ball in $p_v(B)$ around $y$, so $y \notin \partial p_v(B)$.
    Hence no such $x'$ exists, so $p_v$ is injective when restricted to $D_v^0$.
    Since $p_v$ restricted to $D_v^0$ is a continuous injection from a compact set to a Hausdorff space, it is a homeomorphism onto its image (which is the boundary of $p_v(B)$).
\end{proof}

Our aim now is to split the sphere into an upper region $U_v^+$ and a lower region $U_v^-$.

\begin{lemma}\label{interior-pair-of-points}
    Let $x$ be a point in the interior of $p_v(B)$.
    Then there exist exactly two distinct points $x^+, x^- \in S$ with $p_v(x^+) = p_v(x^-) = x$.
\end{lemma}

\begin{proof}
    We know that one such point must exist (as $p_v(S) = p_v(B)$) and if three such points exist then they all lie on the line of the form $x + \lambda v$, so one of the points must lie between the other two.
    This point is then in the interior of $B$ by strict-convexity, which is a contradiction.
    We now show that exactly two such points must exist.
    Consider the line $l$ of the form $x + \lambda v$.
    Viewed as a subset of $\R^{n + 1}$, $l$ is closed, and $l \cap B$ is compact.
    Define the function
    \begin{gather*}
        f : l \cap B \to \R,\\
        f(u) = u \cdot v.
    \end{gather*}
    This function attains a minimum and a maximum in $l \cap B$, and we show that this minimum and maximum are distinct.
    Let $x^+$ be the point at which $f$ attains its maximum, and $x^-$ the minimum.
    Any open ball around $x^{\pm}$ contains a point (itself) in $l$ contained in $B$ and points in $l$ (of the form $x^{\pm} \pm \epsilon v_i$) not contained in $B$.
    Hence $x^{\pm}$ are elements of the boundary of $B$ i.e.\ elements of $S$.
    The point $x^+$ projects to $x$ which is contained in the interior of $p_v(B)$, so is contained in some open ball in the interior of $p_v(B)$.
    Pick points $a, b$ in this open ball such that the line from $a$ to $b$ contains the point $x$.
    Then there exist corresponding points $a', b' \in S$ which project to $a$ and $b$ respectively.
    The line from $a'$ to $b'$ contains a point $x'$ with $p(x') = x$, and by strict-convexity, $x'$ lies in the interior of $B$.
    Hence $x' \neq x^+, x^-$, so $f(x^-) < f(x') < f(x^+)$, implying that $x^- \neq x^+$.
\end{proof}

\begin{figure}[h]
    \centering
    \includegraphics[width=\textwidth]{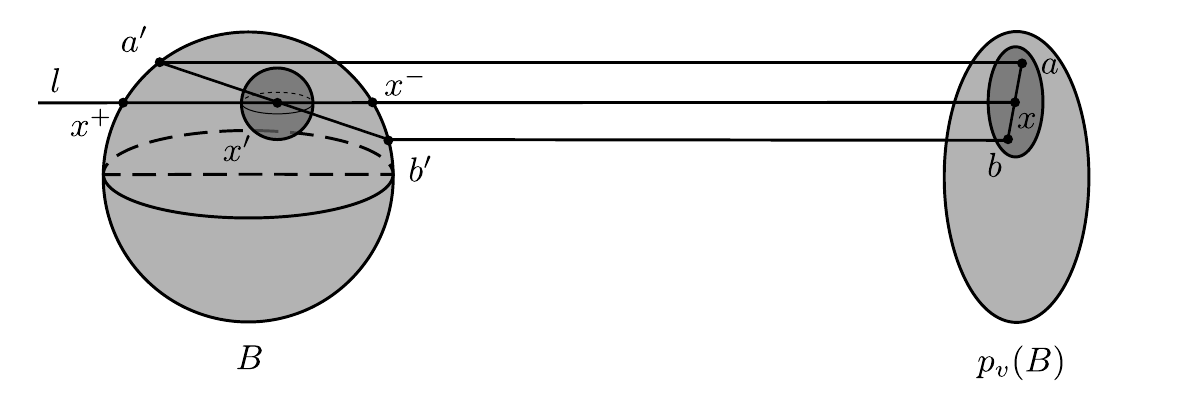}
    \caption{Proof of \autoref{interior-pair-of-points}.}
\end{figure}

We define the sets
\[ U_v^+ = \set{x^+ : x \in \Int p_v(B)}, \qquad U_v^- = \set{x^- : x \in \Int p_v(B)}, \]
which are disjoint.
It is clear that the sets $D_v^{\pm} := D_v^0 \cup U_v^{\pm}$ are in bijection with $p_v(B)$.
Our next aim is to show that $U_v^{\pm}$ are open sets, and therefore the sets $D_v^{\pm}$ are compact.

\begin{lemma}\label{top-and-bottom-open}
    The sets $U_v^+$ and $U_v^-$ are open subsets of $S$.
\end{lemma}

\begin{proof}
    Let $x^+ \in U_v^+, x^- \in U_v^-$ with $p_v(x^+) = p_v(x^-)$.
    Then the point $y = \frac{x^+ + x^-}{2}$ is in the interior of $B$ by strict-convexity.
    Hence there exists $\epsilon > 0$ such that the ball of radius $\epsilon$ around $y$ is contained in the interior of $B$.
    We have $x^+ \neq y$, so let $r := d(x^+, y) > 0$.
    Then for each point $x'$ in the ball of radius $\epsilon$ around $x^+$, the point $x' - rv$ is contained in the ball of radius $\epsilon$ around $y$, and hence in the interior of $B$.
    For $f$ as defined in the previous lemma, we therefore have $f(x') > f(x' - rv)$, so $x' \notin U_v^-$.
    Therefore the open set $V$ defined by the intersection of $S$ with the open ball of radius $\epsilon$ around $x^+$ is disjoint from $U_v^-$.
    Because $S$ is Hausdorff and $D_v^0$ is a compact subset of $S$, we can find an open subset of $S$ containing $x^+$ contained within $V$ that is sufficiently small to be disjoint from $D_v^0$, and is therefore completely contained in $U_v^+$.
    A similar argument holds for $U_v^-$.
\end{proof}

\begin{figure}[htbp]
    \centering
    \includegraphics{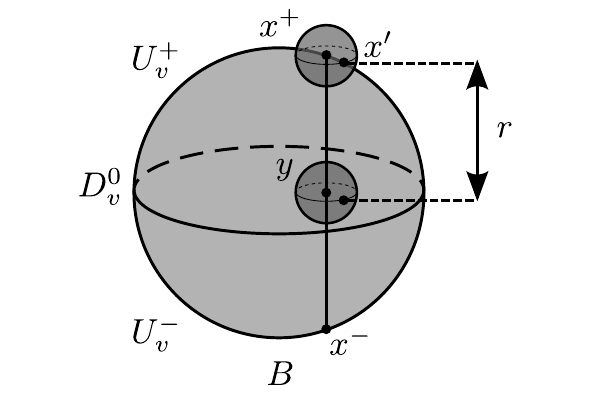}
    \caption{Proof of \autoref{top-and-bottom-open}.}
\end{figure}

Since $D_v^{\pm} = S \setminus U_v^{\mp}$, it follows that the sets $D_v^{\pm}$ are closed, and therefore compact as they are closed subsets of a compact space.
The function $p_v$ restricted to $D_v^{\pm}$ is then a continuous injection from a compact space to a Hausdorff space, so is a homeomorphism onto its image, which is the disc $p_v(B)$.

Define the set
\[ \MM_{v} := \set{\frac{x^+ + x^-}{2} : x^+ \in D_v^+, x^- \in D_v^-, p_v(x^+) = p_v(x^-)}. \]
Echoing Fung's terminology, we refer to $\MM_v$ as the \textit{median} of $S$ in direction $v$.

We are now able to prove \autoref{median-is-disc}. We note the similarities between our functions $p_v, f, g$ and the functions $\pi_{\theta}, f_{\theta}, g_{\theta}$ used in Claim 3.3 of \cite{fung}.

\begin{proof}[Proof (of \autoref{median-is-disc})]
    Define the functions
    \begin{gather*}
        f : D_v^+ \times D_v^- \to \R^n,\\
        f(x^+, x^-) = p_v(x^+) - p_v(x^-),\\
        g : D_v^+ \times D_v^- \to \R^{n + 1},\\
        g(x^+, x^-) = \frac{x^+ + x^-}{2},
    \end{gather*}
    which are both continuous.
    We have $\MM_v = g(f^{-1}(0))$.
    The set $f^{-1}(0)$ is closed as it is the preimage of a closed set, and therefore compact as it is a subset of a compact space.
    Then $g(f^{-1}(0))$ is compact as it is the continuous image of a compact space.
    Hence the restriction of $p_v$ to $\MM_v$ is a continuous injection from a compact space to a Hausdorff space, so is a homeomorphism onto its image, which is the disc $p_v(B)$.
    Furthermore, $\MM_v \cap S = D_v^0$ by strict-convexity, since all other points in $\MM_v$ are midpoints of a distinct pair of points in $B$, so lie in the interior of $B$.
\end{proof}

\section{Intersection of Compact Sets Embedded in $B$}\label{section-3}

Given an orthonormal basis $\set{v_1, \dots, v_{n + 1}}$ of $\R^{n + 1}$, we refer to the projection $p_{v_i}$ as $p_i$, and other objects associated with these vectors (such as $D_{v_i}^0$ and $U_{v_i}^{\pm}$) similarly.
Given $I \subset \set{1, \dots, n + 1}$ with $k := \abs{I} \leq n$, we denote by $p_I : B \to \R^{n + 1 - k}$ the composition of the projections $p_i$ for all $i \in I$.
Note that these projection maps commute (so $p_I$ is independent of the order of composition) and are idempotent (that is, $p_i \circ p_i = p_i$).

The paper \cite{hyperrhombs} refers to a convex subset $B$ of $\R^{n + 1}$ such that
\begin{enumerate}
    \item each point $p$ of the boundary of $B$ has a unique \textit{supporting hyperplane} $\Pi$ ($B$ is disjoint from one of the open half-spaces determined by $\Pi$) containing $p$, and
    \item $\Pi \cap B$ = $\set{p}$,
\end{enumerate}
as \textit{regular}.
An equivalent condition is that $B$ is a strictly-convex subset of $\R^{n + 1}$ with smooth boundary sphere $S$.
The fact that (1) implies smoothness is a consequence of Theorems 23.2 and 25.1 of \cite{rockafellar} (as indicated in \cite{yang}), and the addition of (2) excludes any linear parts of the boundary, implying strict-convexity.

\begin{definition}\label{defn-regularity}
    Let $B$ be the image of a strictly-convex embedding of $\D^{n + 1}$ into $\R^{n + 1}$.
    We say $B$ is \textit{regular} with respect to the orthonormal basis $\set{v_1, \dots, v_{n + 1}}$ of $\R^{n + 1}$ if there is no point $x$ such that $p_i(x) \in \partial p_i(B)$ for all $i \in I$ but $p_I(x) \notin \partial p_I(B)$.
\end{definition}

\autoref{preimage-of-boundary-is-sphere} says that $x \in B$ projects to the boundary of $p_i(B)$ exactly when there are no other elements $y \in B$ with $p_i(x) = p_i(y)$. Any such element $y$ would be of the form $x + \lambda v_i$ for some $\lambda \neq 0$.
Hence an alternative formulation of this condition is: if $x \in B$ such that $x + \lambda_i v_i \notin B$ for all $i \in I, \lambda_i \neq 0$, then $x + \sum_{i \in I} \lambda_i v_i \notin B$ for all $\sum_{i \in I} \abs{\lambda_i} \neq 0$.

\begin{remark}
    This is a generalisation of the regularity presented in \cite{hyperrhombs} in the following sense.
    Let $B$ be the image of an embedding of $\D^{n + 1}$ into $\R^{n + 1}$ bounded by a smoothly embedded sphere $S$, and let $x \in S$.
    If $p_v(x)$ is in the boundary of $p_v(B)$ then \autoref{preimage-of-boundary-is-sphere} tells us that the line $x + \lambda v$ intersects $S$ only at $x$ i.e.\ is tangent to $S$ at $x$.
    Since the tangent hyperplane at $x$ is an affine space, if the lines $x + \lambda v_i$ are tangent to $S$ at $x$ for all $i \in I$, then the tangent hyperplane contains each point of the form $x + \sum_{i \in I} \lambda_i v_i$.
    The tangent hyperplane is the unique supporting hyperplane at $x$ and the only element of $B$ contained in this hyperplane is $x$.
    This is equivalent to the alternative formulation of regularity with respect to a basis, so every regular embedding in the sense of \cite{hyperrhombs} is regular with respect to every basis.
\end{remark}

\begin{definition}\label{defn-special-corner}
    Let $B$ be the image of a strictly-convex embedding of $\D^{n + 1}$ into $\R^{n + 1}$.
    We say $x \in B$ is a \textit{special corner} of $B$ with respect to the basis $\set{v_1, \dots, v_{n + 1}}$ of $\R^{n + 1}$ if $x \in \partial p_i(B)$ for all $i = 1, \dots, n + 1$.
\end{definition}

\begin{remark}
    A special corner of angle $\theta$ of a Jordan curve $\gamma$ in the sense of \cite{fung} is a point $x \in \im \gamma$ such that the lines through $x$ with gradients $\tan\theta$ and $\tan\theta + \frac{\pi}{2}$ intersect $\im \gamma$ only at $x$.
    If we let $B$ be the closed, bounded region determined by $\im \gamma$ and take $v_1 = (\cos\theta, \sin\theta)$ and $v_2 = (\cos \theta + \frac{\pi}{2}, \sin \theta + \frac{\pi}{2})$, then we have $p_1(x) \in \partial p_1(B)$ and $p_2(x) \in \partial p_2(B)$. Therefore \autoref{defn-special-corner} generalises the notion of a special corner in \cite{fung}.
\end{remark}

We note that if the boundary $S$ of $B$ is a smoothly embedded $n$-sphere then $B$ has no special corners with respect to any basis.
Indeed, any such point would have an $(n + 1)$-dimensional tangent space, which is never the case for a smooth embedding of an $n$-manifold such as $S^n$.

From now, we fix an orthonormal basis $\set{v_1, \dots, v_{n + 1}}$ of $\R^{n + 1}$ such that $B$ has no special corners with respect to this basis, and is regular with respect to this basis.

\begin{proof}[Proof (of \autoref{intersections-of-spheres})]
    Let $I$ be some subset of $\set{1, \dots, n + 1}$ of size $k \leq n$.
    Idempotence of the projection maps gives us $p_I(x) = p_I(x) \circ p_i(x)$ for all $i \in I$.
    The set $p_I(B)$ is homeomorphic to $\D^{n + 1 - k}$ by repeated use of \autoref{projection-is-strictly-convex}, and therefore the boundary of $p_I(B)$ is homeomorphic to $S^{n - k}$.
    Suppose $p_I(x)$ is in the boundary of $p_I(B)$ but $x$ is in the interior of $p_i(B)$ for some $i \in I$.
    Note that $p_I \circ p_i(B) = p_I(B)$, but $p_i(B)$ contains an open ball containing $p_i(x)$, which projects to an open ball in $p_I(B)$ containing $p_I(x)$.
    Hence $x$ is not contained in the boundary of $p_I(B)$.
    Since $p_I$ is injective on the boundary of $p_I(B)$ by repeated use of \autoref{preimage-of-boundary-is-sphere}, a copy of $S^{n - k}$ is embedded in $\bigcap_{i \in I} D_i^0$.
    Regularity with respect to the basis implies that every point in this intersection is contained in the embedded copy of $S^{n - k}$, so the embedding is a homeomorphism, proving the first bullet point.

    The function $u \mapsto u \cdot v_j$ attains a maximum and a minimum on $B$, and the points $x^{\pm} \in S$ at which these values are attained are distinct as $B$ does not lie in a hyperplane (otherwise $B$ would have empty interior).
    Suppose some point of the form $x^+ + \lambda v_i$ is contained in $B$ for $i \neq j, \lambda \neq 0$.
    Then by strict-convexity, the point $y := x^+ + \frac{\lambda}{2} v_i$ is contained in the interior of $B$, so $B$ contains an open ball of radius $\epsilon$ around $y$.
    This ball contains the point $y + \frac{\epsilon}{2} v_j$, and $(y + \frac{\epsilon}{2} v_j) \cdot v_j > x^+ \cdot v_j$, which is a contradiction.
    Hence $p_i(x^+)$ is in the boundary of $p_i(B)$ for all $i \neq j$, so $x^+ \in \bigcap_{i \neq j} D_i^0$, and $x^+ \in D_j^+$ as it is the maximum of the function $f$ as defined in \autoref{interior-pair-of-points}.
    The proof is similar for $x^-$.

    The third bullet point is an immediate consequence of the fact that $S$ has no special corners with respect to the basis.
\end{proof}

Any $B$ satisfying the outcomes of \autoref{intersections-of-spheres} is said to satisfy the \textit{crosspolytope condition}.
For each pair $a, b$ of opposite vertices of an $(n + 1)$-crosspolytope, the set of $(n - 1)$-faces of the crosspolytope disjoint from $a$ and $b$ forms the boundary of an $n$-crosspolytope.
We refer to this as a \textit{sub-crosspolytope}, and it is homeomorphic to $S^{n - 1}$.
\begin{figure}[h]
    \centering
    \includegraphics[width=0.8\textwidth]{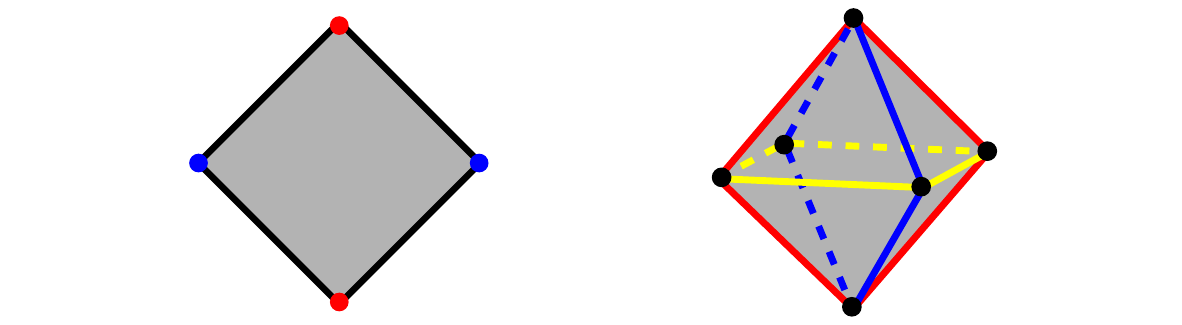}
    \caption{The sub-crosspolytopes of the square and the octahedron are pictured in a single colour.}
\end{figure}

An $(n + 1)$-crosspolytope has $(n + 1)$ pairs $a_1, b_1, \dots, a_{n + 1}, b_{n + 1}$ of opposite vertices, and therefore has $(n + 1)$ sub-crosspolytopes $A_i$.
The intersection of $k$ sub-crosspolytopes is the boundary of an $(n + 1 - k)$-crosspolytope, homeomorphic to $S^{n - k}$.
As a result of the previous proposition, we can find a homeomorphism $h_1$ from the $(n + 1)$-crosspolytope to $B$ such that the sub-crosspolytope $A_i$ is sent to the sphere $D_i^0$.
Each sub-crosspolytope $A_i$ separates the boundary of the crosspolytope into an upper region $B_i$ and a lower region $C_i$, and we have $h_1(B_i) = U_i^+$ and $h_1(C_i) = U_i^-$.
To augment this to a homeomorphism $h$ from the $(n + 1)$-cube to $B$ such that for each pair of opposite faces $B_i', C_i'$ of the $(n + 1)$-cube we have $h(B_i') \subset U_i^+$ and $h(C_i') \subset U_i^-$, we use the following lemma.

\begin{lemma}\label{cube-crosspolytope-balls}
    There exists a homeomorphism $h_2$ which sends the $n$-cube to the $n$-crosspolytope, and for each pair of opposite $(n - 1)$-faces $B_i', C_i'$ of the $n$-cube, we have $h_2(B_i') \subset B_i$ and $h_2(C_i') \subset C_i$.
\end{lemma}

\begin{proof}
    The $n$-crosspolytope is the closed unit ball in the normed space $(\R^n, \norm{\cdot}_1)$, where
    \[ \norm{(x_1, \dots, x_n)}_1 := \sum_{i = 1}^n \abs{x_i}, \]
    and the $n$-cube is the closed unit ball in the normed space $(\R^n, \norm{\cdot}_{\infty})$, where
    \[ \norm{(x_1, \dots, x_n)}_{\infty} := \max(\abs{x_1}, \dots, \abs{x_n}). \]
    We can define a homeomorphism $h_2$ from the crosspolytope to the cube by taking $h_2(x) = \frac{\norm{x}_{\infty}}{\norm{x}_{1}} x$ for $x \neq 0$ and $h_2(0) = 0$.
    Note that $\norm{x}_{\infty} = 1$ implies $\norm{h_2(x)}_1 = 1$ i.e.\ points on the boundary of the cube are sent to points on the boundary of the crosspolytope.
    Opposite $(n - 1)$-faces $B_i', C_i'$ of the cube are the sets
    \[ B_i' = \set{x \in \R^n : x_i = 1}, \qquad C_i' = \set{x \in \R^n : x_i = -1}. \]
    The upper and lower regions $B_i, C_i$ of the boundary of the crosspolytope are the sets
    \[ B_i = \set{x \in \R^n : x_i > 0, \norm{x}_1 = 1}, \qquad C_i = \set{x \in \R^n : x_i < 0, \norm{x}_1 = 1}. \]
    If $x \in B_i'$ then $x_i = 1$ so the $i$th component of $h_2(x)$ is positive, hence $h_2(x_i) \in B_i$.
    The same argument holds in relation to $C_i'$ and $C_i$.
\end{proof}

The desired homeomorphism from the $(n + 1)$-cube to $B$ is then $h := h_1 \circ h_2$.

\begin{figure}[h]
    \centering
    \includegraphics[width=0.85\textwidth]{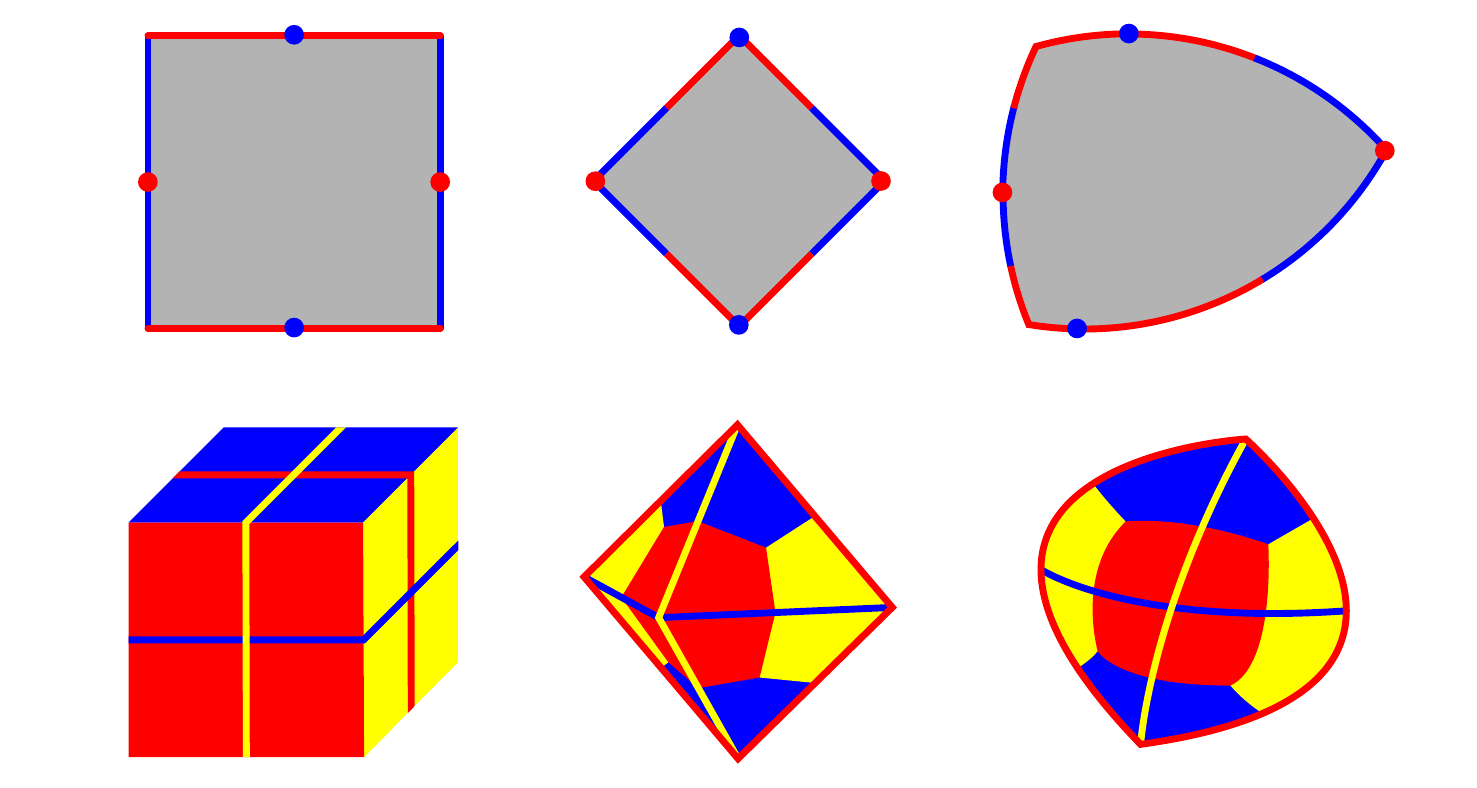}
    \caption{Each sphere $D_i^0$ separates the region of the same colour which corresponds to a pair of opposite $n$-faces of the $(n + 1)$-cube.}
    \label{separates-cube}
\end{figure}

The statement of the Poincar{\'e}-Miranda theorem as extracted from \cite{poincare-miranda} is thus.

\begin{theorem}
    Let $P = [-R_1, R_1] \times \cdots \times [-R_n, R_n]$.
    If $g : P \to \R^n$ is continuous and if, for each $i = 1, \dots, n$,
    \[ g_i(x) \leq 0 \text{ on } \set{x \in P : x_i = -R_i} \text{and } g_i(x) \geq 0 \text{ on } \set{x \in P : x_i = R_i}, \]
    then $g$ has at least one zero in $P$.
\end{theorem}

We remark that the Poincar{\'e}-Miranda theorem is implied by the Borsuk-Ulam theorem as used to prove an analogous result in \cite{hyperrhombs}.

The intuition behind this theorem is that for each component function $g_i : P \to \R$ of $g$, the preimage $g_i^{-1}(0)$ forms a barrier separating two opposite faces of the cube, and these barriers are forced to intersect in a point.
This is analogous to our scenario with medians $\MM_i \cong \D^{n}$ embedded in $B \cong \D^{n + 1}$.
In what follows we work backwards, constructing functions where the corresponding barriers are the medians $\MM_i$, and using the Poincar{\'e}-Miranda theorem to prove that the medians all intersect in at least one point.
In fact, the same reasoning works for any set of $(n + 1)$ compact subsets $K_i$ of $B$ such that there exist pairs of disjoint open subsets $V_i^+, V_i^-$ of $B$ with boundaries contained in $K_i$ such that $U_i^{\pm} \subset V_i^{\pm}$, provided $B$ satisfies the crosspolytope condition (taking $D_i^0 = K_i \cap S$).

The idea behind the following lemma is that if we have an open subset of a metric space, we can replace a continuous function on this open set, as long as we keep the function the same on the boundary of the subset.

\begin{lemma}
    Let $X, Y$ be metric spaces and $f, g : X \to Y$ continuous functions.
    Let $V$ be an open subset of $X$ such that $f$ and $g$ agree on the boundary of $V$.
    Then the function
    \begin{gather*}
        h : X \to Y,\\
        h(x) = \begin{cases}
            f(x)    & \text{if } x \in V,\\
            g(x)    & \text{if } x \notin V,
        \end{cases}
    \end{gather*}
    is continuous.
\end{lemma}

The proof is a simple $\epsilon-\delta$ argument.

Given a compact subset $K$ of a metric space $X$, the \textit{set distance function} $d : X \to \R_{\geq 0}$ which gives the shortest distance from $x \in X$ to any point in $K$ is continuous.
We let $d_i$ be the set distance function of the compact subset $K_i$ of the metric space $B$.
The boundary of $V_i^+$ is contained in $K_i$, and $g_i(x) = 0$ for all $x \in K_i$, so $d_i$ and $-d_i$ agree on the boundary of $V_i^+$.
It then follows from the previous lemma that the functions
\begin{gather*}
    g_i : B \to \R,\\
    g_i(x) = \begin{cases}
        d_i(x)  & \text{if } x \in V_i^+,\\
        -d_i(x)  & \text{if } x \notin V_i^+,\\
    \end{cases}
\end{gather*}
are continuous.
Therefore the function $g : B \to \R^{n + 1}$ with component functions $g_i$ is also continuous.
We note that our use of such functions is similar to that in Proposition 1 of \cite{hyperrhombs}.

\begin{proof}[Proof (of \autoref{intersections-of-manifolds})]
    Since $B$ satisfies the crosspolytope condition we can use \autoref{cube-crosspolytope-balls} to construct the homeomorphism $h$ from the $(n + 1)$-cube to $B$.
    Each component function $g_i$ is positive on $U_i^+$ as $U_i^+ \subset V_i^+$, and negative on $U_i^-$ as $U_i^- \subset V_i^-$.
    Given a pair $B_i', C_i'$ of opposite $n$-faces of the $(n + 1)$-cube, we have $h(B_i') \subset U_i^+$ and $h(C_i') \subset U_i^-$, so $g_i \circ h$ has opposite sign on $B_i'$ and $C_i'$.
    Invoking the Poincar{\'e}-Miranda theorem, we find that $g \circ h(x) = 0$ for some $x \in B$, and this corresponds to an element of $\bigcap_{i} K_i$.
    Any $x \in \bigcap_{i} K_i$ is contained in the interior of $B$ as the intersection of the spheres $D_i^0$ is empty.
\end{proof}

\begin{proof}[Proof (of {\hyperref[main-theorem]{Theorem}})]
    The result follows directly from \autoref{intersections-of-manifolds}, taking $K_i = \MM_i$.
    Indeed, any $x \in \bigcap_{i} \MM_i$ contained in the interior of $B$ is such that $p_i(x)$ is in the interior of $p_i(B)$, hence there exist exactly two points $x_i^{\pm}$ such that $x = \frac{x_i^+ + x_i^-}{2}$.
    The points $x_i^{\pm}$ are the vertices of an $(n + 1)$-rhomb with direction $\set{v_1, \dots, v_{n + 1}}$ inscribed into $S$. 
\end{proof}

\section{Discussion}

This technique could potentially be extended to a proof that the image of every strictly-convex embedding of $S^n$ into $\R^{n + 1}$ inscribes uncountably many $(n + 1)$-rhombs, as in Fung's paper.
In particular, Fung shows that for every Jordan curve there exists an open interval of directions such that the curve inscribes a rhombus in each direction.
This corresponds to an open subset of $S^1 \times S^0$ (if we consider $v$ and $-v$ to be distinct directions).
In the $n$-dimensional case, the orthonormal bases of $\R^{n + 1}$ which determine a special corner correspond to elements of $S^n \times S^{n - 1} \times \cdots \times S^0$, so we pose the following question.
\begin{question*}
    Let $B$ be the image of a strictly-convex embedding of $\D^{n + 1}$ into $\R^{n + 1}$ with boundary $S$.
    Does there exist an open subset $U$ of $S^n \times S^{n - 1} \times \cdots \times S^0$ such that $S$ inscribes an $(n + 1)$-rhomb in every direction contained in $U$?
\end{question*}
The construction in this paper could also be generalised to (non-strictly) convex embeddings.
One particular issue in doing so is that the subsets $D_i^0$ are no longer necessarily homeomorphic to spheres; for the unit square and directions $v_1 = (1, 0)$ and $v_2 = (0, 1)$ the set $D_1^0$ is the union of the top and bottom edges and $D_2^0$ is the union of the left and right edges.

If $B$ is an embedding of $\D^{n + 1}$ into $\R^{n + 1}$, some questions which arise when considering special corners are
\begin{itemize}
    \item Does some $B$ with a countably infinite number of special corners exist? Some $B$ with uncountably many?
    \item Does some $B$ with only finitely many points which are not special corners exist? Countably many?
    \item Given a collection $\mathcal{B}$ of bases of $\R^{n + 1}$, can we construct some $B$ whose special corners are elements $\mathcal{B}$ and only elements of $\mathcal{B}$?
\end{itemize}

\medskip

\bibliographystyle{siam}
\bibliography{rhombs_spheres}

\end{document}